\documentclass[final,reqno]{siamltex}
\usepackage{latexsym,amsmath,amssymb,amsfonts,mathrsfs}
\usepackage{epsf,graphicx,epsfig,color,cite,cases}
\usepackage{subfigure,graphics,multirow,marginnote,enumerate,bm,mathtools}
\newcommand{\Div}{{\rm Div\,}}

\newcommand{\Om}{\Omega}
\newcommand{\om}{\omega}

\newcommand{\pa}{\partial}

\newcommand{\ov}{\overline}
\newcommand{\I}{{\rm Im}}
\newcommand{\Rt}{{\rm Re}}
\newcommand{\curl}{{\rm curl\,}}
\newcommand{\dive}{{\rm div\,}}

\newcommand{\wid}{\widetilde}

\newcommand{\na}{\nabla}
\newcommand{\mat}{\mathbb}

\newcommand{\R}{{\mat R}}

\newcommand{\N}{{\mat N}}
\newcommand{\C}{{\mat C}}

\newcommand{\be}{\begin{eqnarray}}
\newcommand{\ben}{\begin{eqnarray*}}
\newcommand{\en}{\end{eqnarray}}
\newcommand{\enn}{\end{eqnarray*}}

\newtheorem{remark}[theorem]{Remark}

\begin{document}
\renewcommand{\theequation}{\arabic{section}.\arabic{equation}}

\title{\bf
Boundary determination of electromagnetic parameters from local data}
\author{Chengyu Wu\thanks{School of Mathematics and Statistics, Xi'an Jiaotong University,
Xi'an 710049, Shaanxi, China ({\tt wucy99@stu.xjtu.edu.cn})}
\and
Jiaqing Yang\thanks{School of Mathematics and Statistics, Xi'an Jiaotong University,
Xi'an 710049, Shaanxi, China ({\tt jiaq.yang@xjtu.edu.cn})}
}
\date{}
\maketitle

%\vspace{.2in}

\begin{abstract}
  In this paper, we extend and simplify the methods in \cite{CJ24} to improve the results on uniqueness of the boundary determination for the Maxwell equation. In particular, we show that the electromagnetic parameters are uniquely determined to infinite order at the boundary from the local admittance map, disregarding the presence of an unknown obstacle, where actually only the local Cauchy data of the fundamental solution are used. The proof mainly relies on an elaborate singularity analysis on certain singular solutions to the Maxwell equation. 
\end{abstract}



\pagestyle{myheadings}
\thispagestyle{plain}
\markboth{C. Wu and J. Yang}{Boundary determination of electromagnetic parameters from local data}

\section{Introduction}\label{sec1}
\setcounter{equation}{0}

%\section{Statement of the problems and results}\label{sec2}
Let $\Om\subset\R^3$ be a bounded domain with smooth boundary. Denote by $D$ a bounded domain such that $\pa D\in C^\infty$, $D\subset\subset\Om$ and $\Om\setminus\ov D$ is connected. Consider the boundary value problem of finding the electromagnetic fields $E$ and $H$ satisfying 
\be\label{2.1}
\left\{
\begin{array}{ll}
	\curl E-i\om\mu H=0~~~&{\rm in}~\Om\setminus\ov D,\\
	\curl H+i\om\gamma E=0~~~&{\rm in}~\Om\setminus\ov D,\\
	\nu\times E=f~~~&{\rm on}~\pa\Om, \\ 
	\mathcal{B}(E)=0~~~&{\rm on}~\pa D. 
\end{array}
\right.
\en
Here, $\om>0$ is the frequency, $\gamma=\varepsilon+i\sigma/\om$, and the magnetic permeability $\mu$, electric permittivity $\varepsilon$, conductivity $\sigma$ are in $C^\infty(\ov\Om)$. Throughout the paper, we always assume that $\mu,|\gamma|>0$ in $\ov\Om$. Moreover, $\mathcal{B}$ stands for the boundary condition on $\pa D$, which corresponds to a perfect conductor condition if $\mathcal{B}(E)=\nu\times E$ and an impedance condition if $\mathcal{B}(E)=\nu\times\curl E+\rho(\nu\times E)\times\nu$ with $\rho\in C^\infty(\pa D,\C)$. 

Define $H(\curl,\Om\setminus\ov D)$ to be the Hilbert space 
\ben
  H(\curl,\Om\setminus\ov D):=\{u\in[L^2(\Om\setminus\ov D)]^3,\curl u\in[L^2(\Om\setminus\ov D)]^3\} 
\enn
equipped with the inner product $(u,v):=(u,v)_{L^2}+(\curl u,\curl v)_{L^2}$. For $s\in\R$, define 
\ben
  H^s_t(\pa\Om):=\{u\in[H^s(\pa\Om)]^3,u\cdot\nu=0\}. 
\enn
Denote by $\Div$ the surface divergence on $\pa\Om$. Further, define 
\ben
  H^{-1/2}_{\Div}(\pa\Om):=\{u\in H^{-1/2}_t(\pa\Om),\Div u\in H^{-1/2}(\pa\Om)\}. 
\enn

It is known that problem \eqref{2.1} has a unique solution $(E,H)\in H(\curl,\Om\setminus\ov D)\times H(\curl,\Om\setminus\ov D)$ for each $f\in H^{-1/2}_{\Div}(\pa\Om)$ except for a discrete set of magnetic resonance frequencies $\{\om_n\}$. We here assume that $\om$ is not a resonance frequency. Then the admittance map $\Lambda: H^{-1/2}_{\Div}(\pa\Om)\rightarrow H^{-1/2}_{\Div}(\pa\Om)$, 
\ben
  \Lambda: f\mapsto\nu\times H|_{\pa\Om}, 
\enn
is well-defined. In the present paper, we will study the inverse problem of determining the electromagnetic parameters $\mu$ and $\gamma$ from the knowledge of $\Lambda$. In particular, let $\Gamma\subset\pa\Om$ be a nonempty open subset. We shall prove: 
\begin{theorem}\label{thm2.1}
	For $(\mu_i,\varepsilon_i,D_i,\mathcal{B}_i)$, $i=1,2$, suppose $\Lambda_1f=\Lambda_2f$ on $\Gamma$ for all $f\in H^{-1/2}_{\Div}(\pa\Om)$ with ${\rm{supp}}f\subset\Gamma$, then 
    \ben
	 D^\alpha\mu_1=D^\alpha\mu_2,~D^\alpha\gamma_1=D^\alpha\gamma_2~on~\Gamma 
	\enn
	for all $|\alpha|\geq0$. 
\end{theorem}
%If $\mu$ and $\gamma$ are analytic, we further obtain the global uniqueness. 
%\begin{theorem}\label{thm2.2}
%	For $(\mu_i,\varepsilon_i,D_i,\mathcal{B}_i)$, $i=1,2$, suppose $\Lambda_1f=\Lambda_2f$ on $\Gamma$ for all $f\in H^{-1/2}_{\Div}(\pa\Om)$ with ${\rm{supp}}f\subset\Gamma$, and further that $\mu_i,\varepsilon_i$ is analytic in $\ov\Om\setminus D_i$, $i=1,2$, then $D_1=D_2=D$, $\mu_1=\mu_2$ and  $\gamma_1=\gamma_2$ in $\Om\setminus\ov D$, and $\mathcal{B}_1=\mathcal{B}_2$. 
%\end{theorem}

Substantial progresses have been made for the global determination of electromagnetic parameters from complete data, such as \cite{PLE93,PE96}. However, it is assumed in these work that the material parameters are known to some order at the boundary, which is the reason that Joshi and McDowall considered the boundary determination in \cite{MS00,SR97}. Following the idea in \cite{JG89}, they proved that from the complete admittance map the parameters are uniquely determined to infinite order at the boundary by computing the symbol of $\Lambda_1-\Lambda_2$ and $\Lambda_1^{-1}-\Lambda_2^{-1}$, which thus completes the proof of global determination in \cite{PLE93,PE96}. 

Recently, the inverse problem concerning the partial data is gaining more and more interests. For the global determination from partial data we refer to \cite{PPM09,FP18}, where, nevertheless, the global uniqueness results are obtained also assuming the boundary determination results. It is therefore pertinent to further study the boundary determination in partial data case, for which we note that the methods in \cite{MS00,SR97} do not apply to the partial data case, since there are not enough information for $\Lambda^{-1}$ if only knowing partial $\Lambda$. 

In this paper, we extend and simplify the novel method proposed in \cite{CJ24} to give a new proof and slightly improve the boundary determination results in \cite{MS00,SR97} by considering the local admittance map $\Lambda$, and moreover, assuming the presence of an unknown embedded obstacle $(D,\mathcal{B})$. Particularly, we use only the local Cauchy data for the fundamental solution to show that the electromagnetic parameters $\mu$ and $\gamma$ are uniquely determined to infinite order at the boundary, disregarding the unknown obstacle. Our methods mainly depends on the detailed singularity analysis for certain singular solutions and some other relating functions. To be exact, we derive the singularity in the $H^m$ sense for aribitrary $m$ for fixed solutions behaving like $\curl(1/|x-z|)$ and the volume potential of the fundamental solution. Specifically, we utilize the Green's representatioin to derive the full singularity of the solutions, and rigorously study some associated singular integrals to obtain the singularity of the volume potential of the fundamental solution. These explicit singularity allows us to ignore the unknown obstacle and directly compare the singularity of solutions corresponding to different $\mu$ and $\gamma$ locally at the boundary. Any diffenence of the parameters at the boundary would contradict with the singularity obtained previously and thus the boundary determination result follows. We also note that the methods in \cite{MS00,SR97} essentially relied on turning the Maxwell equation to a second order differential equation with the principal term being the Laplacian $\Delta$, while we are working more directly on the Maxwell equation, which indicates that our methods is more handleable when extends to other equations. 

The paper is organized as follows. In section \ref{sec3}, we show some technical results on the singularity to prepare for the proof of the main theorems. In particular, in section \ref{sec3.1} we propose a simple method to derive the complete singularity of the solutions to problem \eqref{2.1} with certain singular boundary data. Then in section \ref{sec3.2}, we compute explicitly the norms of the volume potential of the fundamental solution. Finally, in section \ref{sec4} we prove the main results Theorems \ref{thm2.1}. 
%and \ref{thm2.2}. 

\section{Some preparations}\label{sec3}
\setcounter{equation}{0}
In this section, we do some elaborate singularity analysis for the preparations of the main theorems. We shall derive the complete description of the singularity of the solution to problem (\ref{2.1}) with the boundary data being the fundamental solution to the Maxwell equation. The method mainly depends on the Green's representation of the solutions and thus can be applied to many other cases. Then we simplify the methods in \cite{CJ24} on computing the Sobolev norms at the boundary for the volume potential of fundamental solution, which plays a significant role in later proof.

\subsection{Singular solutions}\label{sec3.1}
We first introduce some useful notations. Let $B_r(x)$ denote the open ball centered at $x\in\R^3$ with radius $r>0$. For balls centered at the origin, we abbreviate by $B_r$. Denote by $\Phi_k(x,y)$ $(\Rt k,\I k\geq0)$ the fundamental solution to $\Delta+k^2$ in $\R^3$, which is 
\ben
\Phi_k(x,y)=\frac{e^{ik|x-y|}}{4\pi|x-y|},~~~\quad x\neq y. 
\enn
Define the single- and double-layer potentials by 
\begin{align*}
(\mathcal{S}_k\varphi)(x)&:=\int_{\pa\Om}\Phi_k(x,y)\varphi(y)ds(y),~~x\in\R^3\setminus\pa\Om, \\
(\mathcal{D}_k\varphi)(x)&:=\int_{\pa\Om}\frac{\pa\Phi_k(x,y)}{\pa\nu(y)}\varphi(y)ds(y),~~x\in\R^3\setminus\pa\Om. 
\end{align*}
Also, we give the definitions of the associated boundary integral operators, for $x\in\pa\Om$, 
\begin{align*}
(S_k\varphi)(x)&:=\int_{\pa\Om}\Phi_k(x,y)\varphi(y)ds(y), \\
(K_k\varphi)(x)&:=\int_{\pa\Om}\frac{\pa\Phi_k(x,y)}{\pa\nu(y)}\varphi(y)ds(y), \\
(K'_k\varphi)(x)&:=\int_{\pa\Om}\frac{\pa\Phi_k(x,y)}{\pa\nu(x)}\varphi(y)ds(y),\\
(M_k\varphi)(x)&:=\nu\times\curl\int_{\pa\Om}\Phi_k(x,y)\varphi(y)ds(y). 
\end{align*}
Further, the volume potential is given by 
\ben
(\mathcal{G}_{k}\varphi)(x):=\int_{\Om}\Phi_k(x,y)\varphi(y)dy,~~~x\in\R^3. 
\enn
When the volume potential (or its derivatives) is restricted on the boundary $\pa\Om$, we denote by $G_{k}$. Since we require $\pa\Om\in C^\infty$, the following mapping properties hold (see details in \cite{GW08,AK89,WM00}). 
\begin{theorem}\label{thm3.1}
	For $s\geq0$, the potentials $\mathcal{S}_k:H^{s-1/2}(\pa\Om)\rightarrow H^{s+1}(\Om)$, $\mathcal{D}_k:H^{s+1/2}(\pa\Om)\rightarrow H^{s+1}(\Om)$ and $\mathcal{G}_{k}:H^s(\Om)\rightarrow H^{s+2}(\Om)$ are bounded. Moreover, the boundary integral operators $S_k,K_k,K'_k:H^{s-1/2}(\pa\Om)\rightarrow H^{s+1/2}(\pa\Om)$ and $M_k:H^{s-1/2}_t(\pa\Om)\rightarrow H^{s+1/2}_t(\pa\Om)$ are bounded for $s\geq0$. 
\end{theorem}
\begin{theorem}\label{thm3.2}
	For $x_0\in\pa\Om$, choose $\delta_0>0$ sufficiently small such that $z_\delta:=x_0+\delta\nu(x_0)\in \R^3\setminus\ov\Om$ for $\delta\in(0,\delta_0)$. Set $\mu_0=\mu(x_0),\gamma_0=\gamma(x_0)$ and $k_0=\om\sqrt{\mu_0\gamma_0}$ $(\Rt k_0,\I k_0\geq0)$. Define $E_{0,\delta}:=\delta^{3/2}\curl(\nu(x_0)\Phi_{k_0}(x,z_\delta))$ and $H_{0,\delta}:=(1/i\om\mu_0)\curl E_{0,\delta}$ for $\delta\in(0,\delta_0)$. Let $(E_\delta,H_\delta)$ be the unique solution to problem \eqref{2.1} with the boundary data $\nu\times E_{0,\delta}$. Set $\varphi_{m,\delta}=\psi_{m,\delta}=\Phi_{m,\delta}=\Psi_{m,\delta}=0$ with $m=-1,0$, 
	\begin{align*}
	  \varphi_{1,\delta}:=&\;2i\om\nu\cdot\curl G_{k_0}(\gamma(\mu-\mu_0)H_{0,\delta})+2\pa_\nu G_{k_0}(\nabla\gamma\cdot E_{0,\delta})~~\qquad{\rm on}~\pa\Om, \\
	  \psi_{1,\delta}:=&\;-2i\om\nu\times\curl G_{k_0}(\mu(\gamma-\gamma_0)E_{0,\delta})+2\nu\times\nabla G_{k_0}(\nabla\mu\cdot H_{0,\delta}) \\
	  &\;-2\nu\times\nabla S_{k_0}(\nu\cdot(\mu-\mu_0)H_{0,\delta})\qquad\qquad\qquad\qquad\qquad\quad~~\;{\rm on}~\pa\Om, 
	\end{align*}
	and 
	\begin{align*}
	  \Phi_{1,\delta}:=&\;i\om\curl\mathcal{G}_{k_0}(\gamma(\mu-\mu_0)H_{0,\delta})+\nabla\mathcal{G}_{k_0}(\nabla\gamma\cdot E_{0,\delta})\qquad\qquad\quad~~{\rm in}~\Om, \\
	  \Psi_{1,\delta}:=&\;-i\om\curl\mathcal{G}_{k_0}(\mu(\gamma-\gamma_0)E_{0,\delta})+\nabla\mathcal{G}_{k_0}(\nabla\mu\cdot H_{0,\delta}) \\
	  &\;-\nabla\mathcal{S}_{k_0}(\nu\cdot(\mu-\mu_0)H_{0,\delta})\qquad\qquad\qquad\qquad\qquad\qquad\qquad~{\rm in}~\Om. 
	\end{align*}
	For $m\in\N$, define 
	\begin{align*}
	   \varphi_{m+1,\delta}:=&\;2\nu\cdot\curl G_{k_0}(\nabla\log\gamma\times\Phi_{m,\delta})+2i\om\nu\cdot\curl G_{k_0}(\gamma\Psi_{m,\delta}) \\
	   &\;-2k_0^2\nu\cdot G_{k_0}(\Phi_{m-1,\delta})+2K_{k_0}'\varphi_{m,\delta}+\varphi_{1,\delta}\qquad\qquad\qquad{\rm on}~\pa\Om, \\
	   \psi_{m+1,\delta}:=&\;2\nu\times\curl G_{k_0}(\na\log\mu\times\Psi_{m,\delta})-2i\om\nu\times\curl G_{k_0}(\mu\Phi_{m,\delta}) \\
	   &\;-2k_0^2\nu\times G_{k_0}(\Psi_{m-1,\delta})-2M_{k_0}\psi_{m,\delta}+\psi_{1,\delta}\qquad\qquad\quad\;\;{\rm on}~\pa\Om,
	\end{align*}
	and 
	\begin{align*}
	  \Phi_{m+1,\delta}:=&\;\curl\mathcal{G}_{k_0}(\nabla\log\gamma\times\Phi_{m,\delta})+i\om\curl\mathcal{G}_{k_0}(\gamma\Psi_{m,\delta})-k_0^2\mathcal{G}_{k_0}(\Phi_{m-1,\delta}) \\
	  &\;+\na\mathcal{S}_{k_0}\varphi_{m+1,\delta}+\Phi_{1,\delta}\qquad\qquad\qquad\qquad\qquad\qquad\qquad\quad{\rm in}~\Om, \\
	  \Psi_{m+1,\delta}:=&\;\curl\mathcal{G}_{k_0}(\na\log\mu\times\Psi_{m,\delta})-i\om\curl\mathcal{G}_{k_0}(\mu\Phi_{m,\delta})-k_0^2\mathcal{G}_{k_0}(\Psi_{m-1,\delta}) \\
	  &\;-\curl\mathcal{S}_{k_0}\psi_{m+1,\delta}+\Psi_{1,\delta}\qquad\qquad\qquad\qquad\qquad\qquad\qquad{\rm in}~\Om. 
	\end{align*}
	Then the following estimates 
	\begin{align*}
	  &\|\nu\cdot\gamma(E_\delta-E_{0,\delta})-\varphi_{m,\delta}\|_{H^{m+1/2}(\pa\Om)}+\|\gamma(E_\delta-E_{0,\delta})-\Phi_{m,\delta}\|_{[H^{m+1}(\Om\setminus\ov D)]^3}\leq C, \\
	  &\|\nu\times\mu(H_\delta-H_{0,\delta})-\psi_{m,\delta}\|_{H^{m+1/2}_t(\pa\Om)}+\|\mu(H_\delta-H_{0,\delta})-\Psi_{m,\delta}\|_{[H^{m+1}(\Om\setminus\ov D)]^3}\leq C 
	\end{align*}
	hold for $m\in\N\cup\{0\}$, where $C>0$ is a constant independent of $\delta\in(0,\delta_0)$. 
\end{theorem}
\begin{proof}
	We first spilt $(E_\delta,H_\delta)$ into two parts. Without loss of genearlity, we assume that $\om$ is also not a resonance frequency for problem \eqref{2.1} when $D=\emptyset$. Then $(\wid E_\delta,\wid H_\delta)$ given by 
	\be\label{3.1}
	\left\{
	\begin{array}{ll}
		\curl\wid E_\delta-i\om\mu\wid H_\delta=0~~~&{\rm in}~\Om,\\
		\curl\wid H_\delta+i\om\gamma\wid E_\delta=0~~~&{\rm in}~\Om,\\
		\nu\times\wid E_\delta=\nu\times E_{0,\delta}~~~&{\rm on}~\pa\Om 
	\end{array}
	\right.
	\en
	is well-defined. Further, it can be verified that $(\hat E_\delta,\hat H_\delta):=(E_\delta,H_\delta)-(\wid E_\delta,\wid H_\delta)$ in $\Om\setminus\ov D$ satisfies 
	\be\label{3.2}
	\left\{
	\begin{array}{ll}
		\curl\hat E_\delta-i\om\mu\hat H_\delta=0~~~&{\rm in}~\Om\setminus\ov D,\\
		\curl\hat H_\delta+i\om\gamma\hat E_\delta=0~~~&{\rm in}~\Om\setminus\ov D,\\
		\nu\times\hat E_\delta=0~~~&{\rm on}~\pa\Om, \\
		\mathcal{B}(\hat E_\delta)=-\mathcal{B}(\wid E_\delta)~~~&{\rm on}~\pa D. 
	\end{array}
	\right.
	\en
	
	Now we claim that for $m\in\N\cup\{0\}$, 
	\begin{align}\label{3.3}
	&\|\nu\cdot\gamma(\wid E_\delta-E_{0,\delta})-\varphi_{m,\delta}\|_{H^{m+1/2}(\pa\Om)}+\|\gamma(\wid E_\delta-E_{0,\delta})-\Phi_{m,\delta}\|_{[H^{m+1}(\Om)]^3}\leq C, \\ \label{3.4}
	&\|\nu\times\mu(\wid H_\delta-H_{0,\delta})-\psi_{m,\delta}\|_{H^{m+1/2}_t(\pa\Om)}+\|\mu(\wid H_\delta-H_{0,\delta})-\Psi_{m,\delta}\|_{[H^{m+1}(\Om)]^3}\leq C 
	\end{align}
	with the constant $C>0$ independent of $\delta\in(0,\delta_0)$. To start with, we rewrite \eqref{3.1} as 
	\be\label{3.5}
	\left\{
	\begin{array}{ll}
		\curl(\wid E_\delta-E_{0,\delta})-i\om\mu(\wid H_\delta-H_{0,\delta})=i\om(\mu-\mu_0)H_{0,\delta}~~~&{\rm in}~\Om,\\
		\curl(\wid H_\delta-H_{0,\delta})+i\om\gamma(\wid E_\delta-E_{0,\delta})=-i\om(\gamma-\gamma_0)E_{0,\delta}~~~&{\rm in}~\Om,\\
		\nu\times(\wid E_\delta-E_{0,\delta})=0~~~&{\rm on}~\pa\Om. 
	\end{array}
	\right.
	\en
	Direct calculation shows that $\|(\mu-\mu_0)H_{0,\delta}\|_{[H^1(\Om)]^3}+\|(\gamma-\gamma_0)E_{0,\delta}\|_{[H^1(\Om)]^3}\leq C$ uniformly for $\delta\in(0,\delta_0)$, which implies by the regularity of Maxwell's equation that $\|\wid E_{\delta}-E_{0,\delta}\|_{[H^1(\Om)]^3}+\|\wid H_\delta-H_{0,\delta}\|_{[H^1(\Om)]^3}\leq C$ uniformly for $\delta\in(0,\delta_0)$. By the trace theorem we then see that \eqref{3.3} and \eqref{3.4} hold for $m=0$. 
	
	From \cite[Theorem 3.25]{AF15}, we have the representations in $\Om$ 
	\begin{align*}
	  \gamma(\wid E_\delta-E_{0,\delta})=&\;\curl\mathcal{G}_{k_0}(\curl(\gamma(\wid E_\delta-E_{0,\delta})))-\na\mathcal{G}_{k_0}(\dive(\gamma(\wid E_\delta-E_{0,\delta}))) \\
	  &\;-k_0^2\mathcal{G}_{k_0}(\gamma(\wid E_\delta-E_{0,\delta}))+\na\mathcal{S}_{k_0}(\nu\cdot\gamma(\wid E_\delta-E_{0,\delta})) 
	\end{align*}
	and 
	\begin{align*}
	  \mu(\wid H_\delta-H_{0,\delta})=&\;\curl\mathcal{G}_{k_0}(\curl(\mu(\wid H_\delta-H_{0,\delta})))-\na\mathcal{G}_{k_0}(\dive(\mu(\wid H_\delta-H_{0,\delta}))) \\
	  &\;-k_0^2\mathcal{G}_{k_0}(\mu(\wid H_\delta-H_{0,\delta}))-\curl\mathcal{S}_{k_0}(\nu\times\mu(\wid H_\delta-H_{0,\delta})) \\
	  &\;+\na\mathcal{S}_{k_0}(\nu\cdot\mu(\wid H_\delta-H_{0,\delta})). 
	\end{align*}
	In view of the equations in \eqref{3.5}, we deduce that 
	\begin{align}\label{3.6}\nonumber
	  &\quad\curl\mathcal{G}_{k_0}(\curl(\gamma(\wid E_\delta-E_{0,\delta}))) \\ \nonumber
	  &=\curl\mathcal{G}_{k_0}(\na\log\gamma\times\gamma(\wid E_\delta-E_{0,\delta}))+\curl\mathcal{G}_{k_0}(\gamma\curl(\wid E_\delta-E_{0,\delta})) \\ \nonumber
	  &=\curl\mathcal{G}_{k_0}(\na\log\gamma\times\gamma(\wid E_\delta-E_{0,\delta}))+i\om\curl\mathcal{G}_{k_0}(\gamma\mu(\wid H_\delta-H_{0,\delta})) \\
	  &\quad+i\om\curl\mathcal{G}_{k_0}(\gamma(\mu-\mu_0)H_{0,\delta})
	\end{align}
	and 
	\begin{align}\label{3.7}\nonumber
	  &\quad\curl\mathcal{G}_{k_0}(\curl(\mu(\wid H_\delta-H_{0,\delta}))) \\ \nonumber
	  &=\curl\mathcal{G}_{k_0}(\na\log\mu\times\mu(\wid H_\delta-H_{0,\delta}))-i\om\curl\mathcal{G}_{k_0}(\mu\gamma(\wid E_\delta-E_{0,\delta})) \\ 
	  &\quad-i\om\curl\mathcal{G}_{k_0}(\mu(\gamma-\gamma_0)E_{0,\delta}). 
	\end{align}
	Further, it is derived that in $\Om$
	\begin{align}\label{3.8}
	  \dive(\gamma(\wid E_\delta-E_{0,\delta}))&=-\dive((\gamma-\gamma_0)E_{0,\delta})=-\na\gamma\cdot E_{0,\delta}, \\ \label{3.9}
	  \dive(\mu(\wid H_\delta-H_{0,\delta}))&=-\dive((\mu-\mu_0)H_{0,\delta})=-\na\mu\cdot H_{0,\delta}. 
	\end{align}
	Moreover, since $\nu\times(\wid E_\delta-E_{0,\delta})=0$ on $\pa\Om$, we have $\nu\cdot\curl(\wid E_\delta-E_{0,\delta})=0$ on $\pa\Om$, 
%	\ben
%	  &&\quad\dive_x(\Phi_{k_0}(x,y)\nu(y)\times(\wid E_\delta(y)-E_{0,\delta}(y))) \\
%	  &&=\nu(y)\cdot\curl_y(\Phi_{k_0}(x,y)(\wid E_\delta(y)-E_{0,\delta}(y)))-\Phi_{k_0}(x,y)\nu(y)\cdot\curl_y(\wid E_\delta(y)-E_{0,\delta}(y)), 
%	\enn
	which indicates from the first equation in \eqref{3.5} that $\nu\cdot\mu(\wid H_\delta-H_{0,\delta})=\nu\cdot(\mu-\mu_0)H_{0,\delta}$ on $\pa\Om$ and thus 
	\be\label{3.10}
	  \na\mathcal{S}_{k_0}(\nu\cdot\mu(\wid H_\delta-H_{0,\delta}))=-\na\mathcal{S}_{k_0}(\nu\cdot(\mu-\mu_0)H_{0,\delta}). 
	\en
	Therefore, combining $\eqref{3.6}-\eqref{3.10}$, we see the representations in $\Om$ become 
	\begin{align}\label{3.11} \nonumber
	  \gamma(\wid E_\delta-E_{0,\delta})=&\;\curl\mathcal{G}_{k_0}(\na\log\gamma\times\gamma(\wid E_\delta-E_{0,\delta}))+i\om\curl\mathcal{G}_{k_0}(\gamma\mu(\wid H_\delta-H_{0,\delta})) \\
	  &\;-k_0^2\mathcal{G}_{k_0}(\gamma(\wid E_\delta-E_{0,\delta}))+\na\mathcal{S}_{k_0}(\nu\cdot\gamma(\wid E_\delta-E_{0,\delta}))+\Phi_{1,\delta}
	\end{align}
	and 
	\begin{align}\label{3.12}\nonumber
	  \mu(\wid H_\delta-H_{0,\delta})=&\;\curl\mathcal{G}_{k_0}(\na\log\mu\times\mu(\wid H_\delta-H_{0,\delta}))-i\om\curl\mathcal{G}_{k_0}(\mu\gamma(\wid E_\delta-E_{0,\delta})) \\
	  &-k_0^2\mathcal{G}_{k_0}(\mu(\wid H_\delta-H_{0,\delta}))-\curl\mathcal{S}_{k_0}(\nu\times\mu(\wid H_\delta-H_{0,\delta}))+\Psi_{1,\delta}. 
	\end{align}
	Restricting to the boundary $\pa\Om$, we further deduce that on $\pa\Om$ 
	\begin{align}\label{3.13}\nonumber
	  &\;\nu\cdot\gamma(\wid E_\delta-E_{0,\delta}) \\ \nonumber
	  =&\;2\nu\cdot\curl G_{k_0}(\na\log\gamma\times\gamma(\wid E_\delta-E_{0,\delta}))+2i\om\nu\cdot\curl G_{k_0}(\gamma\mu(\wid H_\delta-H_{0,\delta})) \\
	  &\;-2k_0^2\nu\cdot G_{k_0}(\gamma(\wid E_\delta-E_{0,\delta}))+2K_{k_0}'(\nu\cdot\gamma(\wid E_\delta-E_{0,\delta}))+\varphi_{1,\delta}
	\end{align}
	and 
	\begin{align}\label{3.14}\nonumber
	  &\;\nu\times\mu(\wid H_\delta-H_{0,\delta}) \\ \nonumber
	  =&\;2\nu\times\curl G_{k_0}(\na\log\mu\times\mu(\wid H_\delta-H_{0,\delta}))-2i\om\nu\times\curl G_{k_0}(\mu\gamma(\wid E_\delta-E_{0,\delta})) \\
	  &\;-2k_0^2\nu\times G_{k_0}(\mu(\wid H_\delta-H_{0,\delta}))-2M_{k_0}(\nu\times\mu(\wid H_\delta-H_{0,\delta}))+\psi_{1,\delta}.
	\end{align}
	
	Recalling the definitions of $\varphi_{m,\delta},\psi_{m,\delta},\Phi_{m,\delta}$ and $\Psi_{m,\delta}$, by the representations \eqref{3.11}-\eqref{3.14}, we obtain that for $m\in\N\cup\{0\}$ 
	\begin{align*}
	&\;\nu\cdot\gamma(\wid E_\delta-E_{0,\delta})-\varphi_{m+1,\delta} \\ 
	=&\;2\nu\cdot\curl G_{k_0}(\na\log\gamma\times[\gamma(\wid E_\delta-E_{0,\delta})-\Phi_{m,\delta}])-2k_0^2\nu\cdot G_{k_0}(\gamma(\wid E_\delta-E_{0,\delta})-\Phi_{m-1,\delta}) \\
	&\;+2i\om\nu\cdot\curl G_{k_0}(\gamma[\mu(\wid H_\delta-H_{0,\delta})-\Psi_{m,\delta}])+2K_{k_0}'(\nu\cdot\gamma(\wid E_\delta-E_{0,\delta})-\varphi_{m,\delta})
	\end{align*}
	and 
	\begin{align*}
	&\;\nu\times\mu(\wid H_\delta-H_{0,\delta})-\psi_{m+1,\delta} \\ 
	=&\;2\nu\times\curl G_{k_0}(\na\log\mu\times[\mu(\wid H_\delta-H_{0,\delta})-\Psi_{m,\delta}])-2M_{k_0}(\nu\times\mu(\wid H_\delta-H_{0,\delta})-\psi_{m,\delta}) \\
	&\;-2i\om\nu\times\curl G_{k_0}(\mu[\gamma(\wid E_\delta-E_{0,\delta})-\Phi_{m,\delta}])-2k_0^2\nu\times G_{k_0}(\mu(\wid H_\delta-H_{0,\delta})-\Psi_{m-1,\delta})
	\end{align*}
	on $\pa\Om$, while in $\Om$ 
	\begin{align*}
	&\;\gamma(\wid E_\delta-E_{0,\delta})-\Phi_{m+1,\delta} \\
	=&\;\curl\mathcal{G}_{k_0}(\na\log\gamma\times[\gamma(\wid E_\delta-E_{0,\delta})-\Phi_{m,\delta}])+i\om\curl\mathcal{G}_{k_0}(\gamma[\mu(\wid H_\delta-H_{0,\delta})-\Psi_{m,\delta}]) \\
	&\;-k_0^2\mathcal{G}_{k_0}(\gamma(\wid E_\delta-E_{0,\delta})-\Phi_{m-1,\delta})+\na\mathcal{S}_{k_0}(\nu\cdot\gamma(\wid E_\delta-E_{0,\delta})-\varphi_{m+1,\delta})
	\end{align*}
	and 
	\begin{align*}
	&\;\mu(\wid H_\delta-H_{0,\delta})-\Psi_{m+1,\delta} \\
	=&\;\curl\mathcal{G}_{k_0}(\na\log\mu\times[\mu(\wid H_\delta-H_{0,\delta})-\Psi_{m,\delta}])-i\om\curl\mathcal{G}_{k_0}(\mu[\gamma(\wid E_\delta-E_{0,\delta})-\Phi_{m,\delta}]) \\
	&-k_0^2\mathcal{G}_{k_0}(\mu(\wid H_\delta-H_{0,\delta})-\Psi_{m-1,\delta})-\curl\mathcal{S}_{k_0}(\nu\times\mu(\wid H_\delta-H_{0,\delta})-\psi_{m+1,\delta}). 
	\end{align*}
	The above four equations clearly confirm the claim \eqref{3.3} and \eqref{3.4} by Theorem \ref{thm3.1} and the induction arguments.
	
	Now by the interior regularity of  Maxwell's equation, we see that for all $m\in\N$, $\|\wid E_\delta\|_{[H^m(D)]^3}\leq C$ uniformly for $\delta$, which implies by the trace theorem that $\|\mathcal{B}(\wid E_\delta)\|_{[H^s(\pa D)]^3}\leq C$ for $s\geq0$. Thus, from \eqref{3.2} we deduce that for all $m\in\N$, $\|\hat E_\delta\|_{[H^m(\Om\setminus\ov D)]^3}+\|\hat H_\delta\|_{[H^m(\Om\setminus\ov D)]^3}\leq C$ uniformly for $\delta$, which leads to the desired estimates and the proof is complete. 
\end{proof}
\begin{remark}\label{re3.3}
	Since $\varphi_{0,\delta}=\psi_{0,\delta}=\Phi_{0,\delta}=\Psi_{0,\delta}=0$, we obtain that for $m\in\N$, $\|\varphi_{m,\delta}\|_{H^{1/2}(\pa\Om)},\|\psi_{m,\delta}\|_{H^{1/2}_t(\pa\Om)},\|\Phi_{m,\delta}\|_{[H^1(\Om)]^3}$ and $\|\Psi_{m,\delta}\|_{[H^1(\Om)]^3}$ are uniformly bounded for $\delta$. Furthermore, by the trace theorem and interior regularity, for $m\in\N$ and $\varepsilon>0$ small we have $\|\varphi_{m,\delta}\|_{H^{m+1/2}(\pa\Om\setminus\ov{B_\varepsilon(x_0)})}+\|\psi_{m,\delta}\|_{H^{m+1/2}_t(\pa\Om\setminus\ov{B_\varepsilon(x_0)})}\leq C$ and $\|\nu\times\Phi_{m,\delta}\|_{H^{m+1/2}_t(\pa\Om)}+\|\nu\cdot\Psi_{m,\delta}+\nu\cdot(\mu-\mu_0)H_{0,\delta}\|_{H^{m+1/2}(\pa\Om)}\leq C$ uniformly for $\delta$. 
\end{remark}

The following result is analogous to Theorem \ref{thm3.2}. The detailed proof is thus omitted. 
\begin{theorem}\label{thm3.4}
	Follow the notations in Theorem {\rm \ref{thm3.2}}. Define $H_{0,\delta}':=\delta^{3/2}$ $\curl(\nu(x_0)\Phi_{k_0}(x,z_\delta))$ and $E_{0,\delta}':=-(1/i\om\gamma_0)\curl H_{0,\delta}'$ for $\delta\in(0,\delta_0)$. Let $(E_\delta',H_\delta')$ be the unique solution to problem \eqref{2.1} with the boundary data $\nu\times E_{0,\delta}'$. Set $\varphi_{m,\delta}'=\psi_{m,\delta}'=\Phi_{m,\delta}'=\Psi_{m,\delta}'=0$ with $m=-1,0$, 
	\begin{align*}
	\varphi_{1,\delta}':=&\;2i\om\nu\cdot\curl G_{k_0}(\gamma(\mu-\mu_0)H_{0,\delta}')+2\pa_\nu G_{k_0}(\nabla\gamma\cdot E_{0,\delta}')~~\qquad{\rm on}~\pa\Om, \\
	\psi_{1,\delta}':=&\;-2i\om\nu\times\curl G_{k_0}(\mu(\gamma-\gamma_0)E_{0,\delta}')+2\nu\times\nabla G_{k_0}(\nabla\mu\cdot H_{0,\delta}') \\
	&\;-2\nu\times\nabla S_{k_0}(\nu\cdot(\mu-\mu_0)H_{0,\delta}')\qquad\qquad\qquad\qquad\qquad\quad~~\;{\rm on}~\pa\Om, 
	\end{align*}
	and 
	\begin{align*}
	\Phi_{1,\delta}':=&\;i\om\curl\mathcal{G}_{k_0}(\gamma(\mu-\mu_0)H_{0,\delta}')+\nabla\mathcal{G}_{k_0}(\nabla\gamma\cdot E_{0,\delta}')\qquad\qquad\quad~~{\rm in}~\Om, \\
	\Psi_{1,\delta}':=&\;-i\om\curl\mathcal{G}_{k_0}(\mu(\gamma-\gamma_0)E_{0,\delta}')+\nabla\mathcal{G}_{k_0}(\nabla\mu\cdot H_{0,\delta}') \\
	&\;-\nabla\mathcal{S}_{k_0}(\nu\cdot(\mu-\mu_0)H_{0,\delta}')\qquad\qquad\qquad\qquad\qquad\qquad\qquad~{\rm in}~\Om. 
	\end{align*}
	For $m\in\N$, define 
	\begin{align*}
	\varphi_{m+1,\delta}':=&\;2\nu\cdot\curl G_{k_0}(\nabla\log\gamma\times\Phi_{m,\delta}')+2i\om\nu\cdot\curl G_{k_0}(\gamma\Psi_{m,\delta}') \\
	&\;-2k_0^2\nu\cdot G_{k_0}(\Phi_{m-1,\delta}')+2K_{k_0}'\varphi_{m,\delta}'+\varphi_{1,\delta}'\qquad\qquad\qquad{\rm on}~\pa\Om, \\
	\psi_{m+1,\delta}':=&\;2\nu\times\curl G_{k_0}(\na\log\mu\times\Psi_{m,\delta}')-2i\om\nu\times\curl G_{k_0}(\mu\Phi_{m,\delta}') \\
	&\;-2k_0^2\nu\times G_{k_0}(\Psi_{m-1,\delta}')-2M_{k_0}\psi_{m,\delta}'+\psi_{1,\delta}'\qquad\qquad\quad\;\;{\rm on}~\pa\Om,
	\end{align*}
	and 
	\begin{align*}
	\Phi_{m+1,\delta}':=&\;\curl\mathcal{G}_{k_0}(\nabla\log\gamma\times\Phi_{m,\delta}')+i\om\curl\mathcal{G}_{k_0}(\gamma\Psi_{m,\delta}')-k_0^2\mathcal{G}_{k_0}(\Phi_{m-1,\delta}') \\
	&\;+\na\mathcal{S}_{k_0}\varphi_{m+1,\delta}'+\Phi_{1,\delta}'\qquad\qquad\qquad\qquad\qquad\qquad\qquad\quad{\rm in}~\Om, \\
	\Psi_{m+1,\delta}':=&\;\curl\mathcal{G}_{k_0}(\na\log\mu\times\Psi_{m,\delta}')-i\om\curl\mathcal{G}_{k_0}(\mu\Phi_{m,\delta}')-k_0^2\mathcal{G}_{k_0}(\Psi_{m-1,\delta}') \\
	&\;-\curl\mathcal{S}_{k_0}\psi_{m+1,\delta}'+\Psi_{1,\delta}'\qquad\qquad\qquad\qquad\qquad\qquad\qquad{\rm in}~\Om. 
	\end{align*}
	Then the following estimates 
	\begin{align*}
	&\|\nu\cdot\gamma(E_\delta'-E_{0,\delta}')-\varphi_{m,\delta}'\|_{H^{m+1/2}(\pa\Om)}+\|\gamma(E_\delta'-E_{0,\delta}')-\Phi_{m,\delta}'\|_{[H^{m+1}(\Om\setminus\ov D)]^3}\leq C, \\
	&\|\nu\times\mu(H_\delta'-H_{0,\delta}')-\psi_{m,\delta}'\|_{H^{m+1/2}_t(\pa\Om)}+\|\mu(H_\delta'-H_{0,\delta}')-\Psi_{m,\delta}'\|_{[H^{m+1}(\Om\setminus\ov D)]^3}\leq C 
	\end{align*}
	hold for $m\in\N\cup\{0\}$, where $C>0$ is a constant independent of $\delta\in(0,\delta_0)$. 
\end{theorem}
\begin{remark}\label{re3.5}
	Similarly, for $m\in\N$ we have that $\|\varphi_{m,\delta}'\|_{H^{1/2}(\pa\Om)}$, $\|\psi_{m,\delta}'\|_{H^{1/2}_t(\pa\Om)}$, $\|\Phi_{m,\delta}'\|_{[H^1(\Om)]^3}$ and $\|\Psi_{m,\delta}'\|_{[H^1(\Om)]^3}$ are uniformly bounded for $\delta$. Further, for $m\in\N$ and $\varepsilon>0$ small we also have $\|\varphi_{m,\delta}'\|_{H^{m+1/2}(\pa\Om\setminus\ov{B_\varepsilon(x_0)})}+\|\psi_{m,\delta}'\|_{H^{m+1/2}_t(\pa\Om\setminus\ov{B_\varepsilon(x_0)})}\leq C$ and $\|\nu\times\Phi_{m,\delta}'\|_{H^{m+1/2}_t(\pa\Om)}+\|\nu\cdot\Psi_{m,\delta}'+\nu\cdot(\mu-\mu_0)H_{0,\delta}'\|_{H^{m+1/2}(\pa\Om)}\leq C$ uniformly for $\delta$. 
\end{remark}

\subsection{The exploding norms}\label{sec3.2}
In this subsection, we show that the Sobolev norms at the boundary of $\na G_{k}(\na q_0\cdot H_{0,\delta})$ would explode for certain $q_0$. In particular, we greatly simplify the method in \cite{CJ24} on analyzing the following singular integral. 
\begin{lemma}\label{lem3.6}
	Follow the notations $x_0,\delta_0$ and $z_\delta$ in Theorem {\rm \ref{thm3.2}}. Suppose $x_0=(0,0,0)$ and $\nu(x_0)=(1,0,0)$. For $m\in\N$ and $\delta\in(0,\delta_0)$, define 
	\ben
	I_{2m-1}(\delta):=\delta^{3/2}\int_{\Om}\frac{1}{|y|}\left(\pa_1^{2m+2}\frac{1}{|y-z_\delta|}\right)y_1^{2m-1}dy, \\
	I_{2m}(\delta):=\delta^{3/2}\int_{\Om}\frac{1}{|y|}\pa_1\left[\left(\pa_1^{2m+2}\frac{1}{|y-z_\delta|}\right)y_1^{2m}\right]dy. 
	\enn
	Then for arbitrarily fixed $m\in\N$, $|I_m(\delta)|\rightarrow+\infty$ as $\delta\rightarrow0^+$. 
\end{lemma}
\begin{proof}
	Since $\pa\Om\in C^\infty$, without loss of genearlity we can assume that $\pa\Om$ is flat near $x_0$. In particular, we suppose that $\Om=\{y\in\R^3,y_2^2+y_3^2<1,-1<y_1<0\}$. 
	
	To prove the lemma, we first define some auxiliary functions for $\delta$. For $m\in\N$ and $\delta\in(0,\delta_0)$, define 
	\ben
	  \wid I_m(\delta):=\int_{\Om}\frac{1}{|y|}\left(\pa_1\frac{1}{|y-z_\delta|}\right)y_1^{m}dy. 
	\enn
	By the explicit expression of $\Om$, we deduce that 
	\begin{align*}
	  \wid I_m(\delta)=&\;-\int_{\Om}\frac{1}{|y|}\frac{y_1-\delta}{|y-z_\delta|^3}y_1^{m}dy \\
	  =&\;(-1)^m\pi\int_{0}^{1}\int_{0}^{1}\frac{1}{\sqrt{y_1^2+r}}\frac{y_1+\delta}{\sqrt{(y_1+\delta)^2+r}^3}y_1^mdrdy_1 \\
	  =&\;(-1)^m2\pi\int_{0}^{1}(y_1+\delta)y_1^m\left(\int_{y_1}^{\sqrt{y_1^2+1}}\frac{1}{\sqrt{s^2+(y_1+\delta)^2-y_1^2}^3}ds\right)dy_1. 
	\end{align*}
	It is known that 
	\ben
	  \int_{y_1}^{\sqrt{y_1^2+1}}\frac{1}{\sqrt{s^2+(y_1+\delta)^2-y_1^2}^3}ds=\frac{1}{\delta(2y_1+\delta)}\left(\frac{\sqrt{y_1^2+1}}{\sqrt{(y_1+\delta)^2+1}}-1+\frac{\delta}{y_1+\delta}\right), 
	\enn
	which implies 
	\ben
	  \wid I_m(\delta)=(-1)^m2\pi\int_{0}^{1}\left(\frac{y_1^m}{2y_1+\delta}+\frac{-(y_1+\delta)y_1^m}{\sqrt{(y_1+\delta)^2+1}(\sqrt{y_1^2+1}+\sqrt{(y_1+\delta)^2+1})}\right)dy_1. 
	\enn
	Direct calculation shows that 
	\begin{align*}
	  (-1)^m2\pi\int_{0}^{1}\frac{y_1^m}{2y_1+\delta}dy_1=&\;(-1)^m\pi\int_{0}^{1}\sum_{j=0}^{m}C_m^j(y_1+\frac{\delta}{2})^{j-1}(-\frac{\delta}{2})^{m-j}dy_1 \\
	  =&\;-\frac{\pi}{2^m}\delta^m\ln\delta+\left[(-\frac{\delta}{2})^m\ln(2\delta+1)\right. \\
	  &\;\left.+\sum_{j=1}^{m}(-\frac{\delta}{2})^{m-j}\frac{C_m^j}{j}\left((1+\frac{\delta}{2})^j-(\frac{\delta}{2})^j\right)\right]. 
	\end{align*}
	Therefore, we derive that for $m\in\N$ and $\delta\in(0,\delta_0)$ 
	\ben
	  \wid I_m(\delta)=-\frac{\pi}{2^m}\delta^m\ln\delta+f_m(\delta), 
	\enn
	where $f_m$ is a smooth function in $[0,\delta_0]$. 
	
	Now we investigate on $I_m(\delta)$. It is seen that for $m\in\N$ and $\delta\in(0,\delta_0)$ 
	\begin{align*}
	  \int_{\Om}\frac{1}{|y|}\left(\pa_1^{m+3}\frac{1}{|y-z_\delta|}\right)y_1^{m}dy=&\;(-1)^m\int_{\Om}\frac{1}{|y|}\left(\pa_\delta^{m+2}\pa_1\frac{1}{|y-z_\delta|}\right)y_1^{m}dy \\
	  =&\;(-1)^m\frac{d^{m+2}}{d\delta^{m+2}}\wid I_m(\delta) \\
	  =&\;(-1)^m\pi\frac{m!}{2^m}\delta^{-2}+\frac{d^{m+2}}{d\delta^{m+2}}f_m(\delta), 
	\end{align*}
	which implies that 
	\ben
	  &&I_{2m-1}(\delta)=-\pi\frac{(2m-1)!}{2^{2m-1}}\delta^{-1/2}+\delta^{3/2}\frac{d^{2m+1}}{d\delta^{2m+1}}f_{2m-1}(\delta), \\
	  &&I_{2m}(\delta)=-\pi\frac{(2m)!}{2^{2m}}\delta^{-1/2}+\delta^{3/2}\left(2m\frac{d^{2m+1}}{d\delta^{2m+1}}f_{2m-1}(\delta)+\frac{d^{2m+2}}{d\delta^{2m+2}}f_{2m}(\delta)\right) 
	\enn
	for $m\in\N$ and $\delta\in(0,\delta_0)$. Since $f_m$ is smooth in $[0,\delta_0]$, we obtain the conclusion. 
\end{proof}
\begin{theorem}\label{thm3.7}
	Follow the notations $x_0,\delta_0$ and $z_\delta$ in Theorem {\rm \ref{thm3.2}}. For fixed $m\in\N\cup\{0\}$, let $q\in C^\infty(\ov\Om)$ be such that $D^\alpha q=0$ on $\Gamma_\varepsilon:=\pa\Om\cap B_\varepsilon(x_0)$ with $|\alpha|\leq m$ and $\varepsilon>0$ small, and $\pa_\nu^{m+1}q(x_0)\neq0$, then we have 
	\ben
	\left\|\nu\times\na G_k(\na q\cdot H_{0,\delta})\right\|_{H_t^{m+3/2}(\pa\Om)}+\left\|\nu\cdot\na G_k(\na q\cdot H_{0,\delta})\right\|_{H^{m+3/2}(\pa\Om)}\rightarrow+\infty
	\enn
	as $\delta\rightarrow0^+$. 
\end{theorem}
\begin{proof}
	Without loss of genearlity, we suppose that $x_0=(0,0,0)$ and $\nu(x_0)=(1,0,0)$. From the definition of $H_{0,\delta}$, it is noted that $H_{0,\delta}=\delta^{3/2}(k_0^2\Phi_{k_0}(\cdot,z_\delta)+\pa_1^2\Phi_{k_0}(\cdot,z_\delta),\pa_{12}\Phi_{k_0}(\cdot,z_\delta),\pa_{13}\Phi_{k_0}(\cdot,z_\delta))$. By the assumption on $q$, we know that $D^\alpha q(x_0)=0$ for $\alpha=(\alpha_1,\alpha_2,\alpha_3)$ with $\alpha_1\leq m$, which implies that $\delta^{3/2}\pa_1q\Phi_{k_0}(\cdot,z_\delta)$, $\delta^{3/2}\pa_2q\pa_{12}\Phi_{k_0}(\cdot,z_\delta)$ and $\delta^{3/2}\pa_3q\pa_{13}\Phi_{k_0}(\cdot,z_\delta)$ are uniformly bounded in $H^{m+1}(\Om)$ for $\delta\in(0,\delta_0)$. Therefore, it suffices to show that 
	\ben
	&&\left\|\nu\times\na G_k(\delta^{3/2}\pa_1q\pa_1^2\Phi_{k_0}(\cdot,z_\delta))\right\|_{H_t^{m+3/2}(\pa\Om)} \\
	&&\qquad+\left\|\nu\cdot\na G_k(\delta^{3/2}\pa_1q\pa_1^2\Phi_{k_0}(\cdot,z_\delta))\right\|_{H^{m+3/2}(\pa\Om)}\rightarrow+\infty
	\enn
	as $\delta\rightarrow0^+$. Further, from the regularity of $\Phi_{k}-\Phi_0$, it is enough to prove that 
	\be\nonumber
	&&\left\|\nu\times\na G_0(\delta^{3/2}\pa_1q\pa_1^2\Phi_{0}(\cdot,z_\delta))\right\|_{H_t^{m+3/2}(\pa\Om)} \\ \label{3.15}
	&&\qquad+\left\|\nu\cdot\na G_0(\delta^{3/2}\pa_1q\pa_1^2\Phi_{0}(\cdot,z_\delta))\right\|_{H^{m+3/2}(\pa\Om)}\rightarrow+\infty
	\en
	as $\delta\rightarrow0^+$. 
	
	Now consider the case that $m=0$. We prove by contradiction. Assume there exists a sequence $\{\delta_l\}_{l\in\N}\subset(0,\delta_0)$ with $\delta_l\rightarrow0$ as $l\rightarrow+\infty$ such that the norms in \eqref{3.15} are bounded for $\{\delta_l\}$. It is known that $\delta^{3/2}(\pa_1q-\pa_1q(x_0))\pa_1^2\Phi_{0}(\cdot,z_\delta)$ are uniformly bounded in $H^1(\Om)$ for $\delta\in(0,\delta_0)$, which implies by Theorem \ref{thm3.1} and the trace theorem that $\na G_0[\delta^{3/2}(\pa_1q-\pa_1q(x_0))\pa_1^2\Phi_{0}(\cdot,z_\delta)]$ are uniformly bounded in $[H^{3/2}(\pa\Om)]^3$ for $\delta$. Since $\pa_1q(x_0)\neq0$, we deduce that 
	\be\label{3.16}
	  \left\|\na G_0(\delta_l^{3/2}\pa_1^2\Phi_{0}(\cdot,z_{\delta_l}))\right\|_{[H^{3/2}(\pa\Om)]^3}\leq C
	\en
    uniformly for $l\in\N$. From the boundedness of $\delta_l^{3/2}\pa_1^2\Phi_{0}(\cdot,z_{\delta_l})$ in $L^2(\Om)$, it is obtained that $G_0(\delta_l^{3/2}\pa_1^2\Phi_{0}(\cdot,z_{\delta_l}))$ is uniformly bounded in $H^{3/2}(\pa\Om)$ and hence in $H^{5/2}(\pa\Om)$ by \eqref{3.16}. Further, we note that $\Delta^2G_0(\delta_l^{3/2}\pa_1^2\Phi_{0}(\cdot,z_{\delta_l}))=0$ in $\Om$. Then by the regularity of biharmonic equation (see details in \cite{FHG10}) we derive that $G_0(\delta_l^{3/2}\pa_1^2\Phi_{0}(\cdot,z_{\delta_l}))$ are uniformly bounded in $H^3(\Om)$, which indicates that $\delta_l^{3/2}\pa_1^2\Phi_{0}(\cdot,z_{\delta_l})$ are uniformly bounded in $H^1(\Om)$ for $l\in\N$. This is however a contradiction and thus the conclusion follows. 
    
    We finish the proof by applying Lemma \ref{lem3.6}. Suppose the oppisite that the norms are bounded for $\{\delta_l\}$. Clearly, the norms are also bounded when restricted on $\Gamma_\varepsilon$, i.e., 
    \ben
   &&\left\|\nu\times\na G_0(\delta_l^{3/2}\pa_1q\pa_1^2\Phi_{0}(\cdot,z_{\delta_l}))\right\|_{H_t^{m+3/2}(\Gamma_\varepsilon)} \\ 
   &&\qquad+\left\|\nu\cdot\na G_0(\delta_l^{3/2}\pa_1q\pa_1^2\Phi_{0}(\cdot,z_{\delta_l}))\right\|_{H^{m+3/2}(\Gamma_\varepsilon)}\leq C. 
    \enn
    Since $\pa\Om\in C^\infty$, for $\varepsilon>0$ sufficiently small, we can smoothly straighten the boundary $\Gamma_\varepsilon$ such that the norms after straightening are equivalent (see \cite[Theorem 3.23]{WM00}). Therefore, there is no loss of genearlity if we assume that $\Gamma_\varepsilon\subset\{y\in\R^3,y_1=0\}$. In other words, we have $\nu=(1,0,0)$ on $\Gamma_\varepsilon$. 
    
    We first consider the case that $m=2i_0-1$ with $i_0\in\N$. From the condition that $\pa_1q=0$ on $\Gamma_\varepsilon$, integration by parts yields that 
    \begin{align*}
    \pa_j\int_{\Om}\Phi_{0}(x,y)\pa_1q(y)\pa_1^2\Phi_{0}(y,z_{\delta_l})dy=&\;\int_{\Om}\Phi_{0}(x,y)\pa_j\left(\pa_1q(y)\pa_1^2\Phi_{0}(y,z_{\delta_l})\right)dy \\
    &\;-\int_{\pa\Om\setminus\ov\Gamma_\varepsilon}\Phi_{0}(x,y)\pa_1q(y)\pa_1^2\Phi_{0}(y,z_{\delta_l})\nu_j(y)ds(y)
    \end{align*}
    for $x\in\ov\Om$ and $j=1,2,3$. Evidently, we have 
    \ben
    \left\|\int_{\pa\Om\setminus\ov\Gamma_\varepsilon}\Phi_{0}(x,y)\delta_l^{3/2}\pa_1q(y)\pa_1^2\Phi_{0}(y,z_{\delta_l})\nu_j(y)ds(y)\right\|_{H^{m+3/2}(\Gamma_\varepsilon)}\leq C 
    \enn
    uniformly for $l\in\N$. Since $D^\alpha q(x_0)=0$ for $\alpha=(\alpha_1,\alpha_2,\alpha_3)$ with $\alpha_1\leq m$, for $j=2,3$ we deduce that $\delta_l^{3/2}\pa_j\pa_1q(y)\pa_1^2\Phi_{0}(y,z_{\delta_l})$ are uniformly bounded in $H^m(\Om)$. Thus, we derive that 
   \ben
     \left\|\int_{\Om}\Phi_{0}(x,y)\delta_l^{3/2}\pa_1q(y)\pa_{j_1}\pa_1^2\Phi_{0}(y,z_{\delta_l})dy\right\|_{H^{m+3/2}(\Gamma_\varepsilon)}\leq C 
   \enn
   for $j_1=2,3$ and $l\in\N$, which leads to 
   \ben
      \left\|\nu\times\na\int_{\Om}\Phi_{0}(x,y)\delta_l^{3/2}\pa_1q(y)\pa_{j_1}\pa_1^2\Phi_{0}(y,z_{\delta_l})dy\right\|_{H^{m+1/2}_t(\Gamma_\varepsilon)}\leq C. 
   \enn
   Analyzing similarly, we obtain that 
    \ben
   \left\|\int_{\Om}\Phi_{0}(x,y)\delta_l^{3/2}\pa_1q(y)\pa_{j_1}\pa_{j_2}\pa_1^2\Phi_{0}(y,z_{\delta_l})dy\right\|_{H^{m+1/2}(\Gamma_\varepsilon)}\leq C 
   \enn
   for $j_1,j_2=2,3$. By taking $(j_1,j_2)=(2,2),(3,3)$, since $\pa_1^2\Phi_{0}(\cdot,z_{\delta_l})$ is harmonic in $\Om$, it further yields that 
   \ben
   \left\|\int_{\Om}\Phi_{0}(x,y)\delta_l^{3/2}\pa_1q(y)\pa_1^4\Phi_{0}(y,z_{\delta_l})dy\right\|_{H^{m+1/2}(\Gamma_\varepsilon)}\leq C. 
   \enn
   Repeating the above procedure $i_0-1$ times, finally we derive that 
   \ben
   \left\|\int_{\Om}\Phi_{0}(x,y)\delta_l^{3/2}\pa_1q(y)\pa_1^{2i_0+2}\Phi_{0}(y,z_{\delta_l})dy\right\|_{H^{3/2}(\Gamma_\varepsilon)}\leq C,  
   \enn
   which implies by the bounded embedding from $H^{3/2}(\Gamma_\varepsilon)$ into $C(\Gamma_\varepsilon)$ that 
   \ben
   \left|\int_{\Om}\frac{1}{|y|}\delta_l^{3/2}\pa_1q(y)\pa_1^{2i_0+2}\frac{1}{|y-z_{\delta_l}|}dy\right|\leq C. 
   \enn
   We further yield by the Taylor's expansion that 
   \ben
   \left|\delta_l^{3/2}\left(\pa_{1}q_0(y)-\frac{\pa_1^{2i_0}q_0(x_0)}{(2i_0-1)!}y_1^{2i_0-1}\right)\pa_{1}^{2i_0+2}\frac{1}{|y-z_{\delta_l}|}\right|\leq\frac{C}{|y-z_{\delta_l}|^{3/2}},~y\in\Om. 
   \enn
   Since $\pa_{1}^{m+1}q_0(x_0)=\pa_{1}^{2i_0}q_0(x_0)\neq0$, we obtain that 
   \ben
   \left|\delta_l^{3/2}\int_{\Om}\frac{1}{|y|}y_1^{2i_0-1}\pa_{1}^{2i_0+2}\frac{1}{|y-z_{\delta_l}|}dy\right|\leq C 
   \enn
   uniformly for $l\in\N$, which clearly contradicts with Lemma \ref{lem3.6}. Therefore, we prove the case when $m$ is odd. 
   
   As for the case that $m=2i_0$ with $i_0\in\N$, from the estimate that 
   \ben
     \left\|\nu\cdot\na G_0(\delta_l^{3/2}\pa_1q\pa_1^2\Phi_{0}(\cdot,z_{\delta_l}))\right\|_{H^{m+3/2}(\Gamma_\varepsilon)}\leq C, 
   \enn
   through the same methods, we can deduce that 
   \ben
     \left|\delta_l^{3/2}\int_{\Om}\frac{1}{|y|}\pa_1\left[\left(\pa_1^{2m+2}\frac{1}{|y-z_{\delta_l}|}\right)y_1^{2m}\right]dy\right|\leq C, 
   \enn
   which contradicts with Lemma \ref{lem3.6} again. The proof is thus finished. 
\end{proof}

Recalling the definition of $H_{0,\delta}$ and $E_{0,\delta}'$, we immediately have the following result. 
\begin{theorem}\label{thm3.8}
	Follow the notations $x_0,\delta_0$ and $z_\delta$ in Theorem {\rm \ref{thm3.2}}. For fixed $m\in\N\cup\{0\}$, let $q\in C^\infty(\ov\Om)$ be such that $D^\alpha q=0$ on $\Gamma_\varepsilon:=\pa\Om\cap B_\varepsilon(x_0)$ with $|\alpha|\leq m$ and $\varepsilon>0$ small, and $\pa_\nu^{m+1}q(x_0)\neq0$, then we have 
	\ben
	\left\|\nu\times\na G_k(\na q\cdot E_{0,\delta}')\right\|_{H_t^{m+3/2}(\pa\Om)}+\left\|\nu\cdot\na G_k(\na q\cdot E_{0,\delta}')\right\|_{H^{m+3/2}(\pa\Om)}\rightarrow+\infty
	\enn
	as $\delta\rightarrow0^+$. 
\end{theorem}

\section{Proofs of main results}\label{sec4}
\setcounter{equation}{0}
In this section, we first study some further properties for the functions defined in Theorems \ref{thm3.2} and \ref{thm3.4}. Then we give the proof of the main results, i.e., Theorems \ref{thm2.1}. 
%and \ref{thm2.2}. 

In this section, we always assume that the condition in Theorem \ref{thm2.1} holds, i.e., $\Lambda_1=\Lambda_2$ locally on $\Gamma$. 
\begin{lemma}\label{lem4.1}
	$\mu_1=\mu_2$ and $\gamma_1=\gamma_2$ on $\Gamma$. 
\end{lemma}
\begin{proof}
	Suppose there exists a $x_0\in\Gamma$ such that $\mu_1(x_0)\neq\mu_2(x_0)$. Set $\mu_0^{(i)}=\mu_i(x_0)$ and $\gamma_0^{(i)}=\gamma_i(x_0)$, $i=1,2$. Follow the notations $\delta_0$ and $z_\delta$ in Theorem \ref{thm3.2}. Define $U_{0,\delta}:=\delta^{1/2}\curl(\nu(x_0)\Phi_0(\cdot,z_\delta))$ and $V_{0,\delta}^{(i)}=(1/i\om\mu_0^{(i)})\curl U_{0,\delta}$ for $i=1,2$ and $\delta\in(0,\delta_0)$. Let $(U_\delta^{(i)},V_\delta^{(i)})$ be the unique solution to problem \eqref{2.1} with respect to the boudnary data $\nu\times U_{0,\delta}$ and $(\mu_i,\varepsilon_i,D_i,\mathcal{B}_i)$, $i=1,2$. Then it can be verified that for $i=1,2$, 
	\ben
	\left\{
	\begin{array}{ll}
		\curl(U_\delta^{(i)}-U_{0,\delta})-i\om\mu_i(V_\delta^{(i)}-V_{0,\delta}^{(i)})=i\om(\mu_i-\mu_0^{(i)})V_{0,\delta}^{(i)}~~~&{\rm in}~\Om\setminus\ov D_i,\\
		\curl(V_\delta^{(i)}-V_{0,\delta}^{(i)})+i\om\gamma_i(U_\delta^{(i)}-U_{0,\delta})=-i\om\gamma_i U_{0,\delta}~~~&{\rm in}~\Om\setminus\ov D_i,\\
		\nu\times(U_\delta^{(i)}-U_{0,\delta})=0~~~&{\rm on}~\pa\Om, \\
		\mathcal{B}_i(U_\delta^{(i)}-U_{0,\delta})=-\mathcal{B}_i(U_{0,\delta})~~~&{\rm on}~\pa D_i. 
	\end{array}
	\right.
	\enn
	By the regularity of Maxwell's equation, we thus obtain that 
	$
	  \|U_\delta^{(i)}-U_{0,\delta}\|_{H(\curl,\Om\setminus\ov D_i)}+\|V_\delta^{(i)}-V_{0,\delta}^{(i)}\|_{H(\curl,\Om\setminus\ov D_i)}\leq C
	$
	with $i=1,2$ and uniformly for $\delta$. For $x_0$, we can find a small $\varepsilon>0$ such that $\pa\Om\cap B_{2\varepsilon}(x_0)\subset\Gamma$ and $B_{2\varepsilon}(x_0)\cap(D_1\cup D_2)=\emptyset$. Let $\eta\in C^\infty(\R^3)$ be the cut-off function satisfying that $\eta=1$ in $B_\varepsilon(x_0)$ and $\eta=0$ in $\R^3\setminus\ov{B_{2\varepsilon}(x_0)}$. Define $(\wid U_\delta^{(i)},\wid V_\delta^{(i)})$, $i=1,2$, to be the unique solution to problem \eqref{2.1} with the boundary data $\nu\times \eta U_{0,\delta}$. Still, we have the estimates that 
	\be\label{4.1}
	  \|\wid U_\delta^{(i)}-U_{0,\delta}\|_{H(\curl,\Om\setminus\ov D_i)}+\|\wid V_\delta^{(i)}-V_{0,\delta}^{(i)}\|_{H(\curl,\Om\setminus\ov D_i)}\leq C
	\en
	uniformly for $\delta$. Now further choose a $C^2$ domain $D_0\subset\Om$ such that $B_{2\varepsilon}(x_0)\cap\Om\subset D_0$, $D_0\cap(D_1\cup D_2)=\emptyset$ and $\pa\Om\cap\pa D_0\subset\Gamma$. Note that $\nu\times\wid V_\delta^{(1)}=\nu\times\wid V_\delta^{(2)}$ on $\pa\Om\cap\pa D_0$. Integration by parts over $D_0$ then yields that 
	\ben
	  \int_{D_0}\gamma_2\wid U_\delta^{(1)}\cdot\wid U_\delta^{(2)}dx-\frac{1}{i\om}\int_{\pa D_0}\nu\times\wid U_\delta^{(1)}\cdot[(\nu\times\wid V_\delta^{(2)})\times\nu]ds+\int_{D_0}\mu_1\wid V_\delta^{(1)}\cdot\wid V_\delta^{(2)}dx=0, \\
	  \int_{D_0}\gamma_1\wid U_\delta^{(1)}\cdot\wid U_\delta^{(2)}dx-\frac{1}{i\om}\int_{\pa D_0}\nu\times\wid U_\delta^{(2)}\cdot[(\nu\times\wid V_\delta^{(1)})\times\nu]ds+\int_{D_0}\mu_2\wid V_\delta^{(1)}\cdot\wid V_\delta^{(2)}dx=0, 
	\enn
	which implies that 
	\begin{align*}
	  &\;\int_{D_0}(\mu_1-\mu_2)\wid V_\delta^{(1)}\cdot\wid V_\delta^{(2)}dx \\
	  =&\;\int_{D_0}(\gamma_1-\gamma_2)\wid U_\delta^{(1)}\cdot\wid U_\delta^{(2)}dx+\frac{1}{i\om}\int_{\pa D_0\setminus\pa\Om}\nu\times\wid U_\delta^{(1)}\cdot[(\nu\times\wid V_\delta^{(2)})\times\nu]ds \\
	  &\;-\frac{1}{i\om}\int_{\pa D_0\setminus\pa\Om}\nu\times\wid U_\delta^{(2)}\cdot[(\nu\times\wid V_\delta^{(1)})\times\nu]ds. 
	\end{align*}
	From the estimats \eqref{4.1}, we deduce that 
	\ben
	  \left|\int_{D_0}(\mu_1-\mu_2)V_{0,\delta}^{(1)}\cdot V_{0,\delta}^{(2)}dx\right|\leq C
	\enn
	uniformly for $\delta\in(0,\delta_0)$. However, by the singularity of $V_{0,\delta}^{(i)}$, this contradicts with $\mu_1(x_0)\neq\mu_2(x_0)$ as $\delta\rightarrow0^+$. Therefore, $\mu_1=\mu_2$ on $\Gamma$. 
	
	Now suppose $\gamma_1(x_0)\neq\gamma_2(x_0)$ for $x_0\in\Gamma$. Define $U_{0,\delta}'=\delta^{1/2}\curl\curl(\nu(x_0)$ $\Phi_0(\cdot,z_\delta))$ and $V_{0,\delta}^{(i)\prime}=i\om\gamma_0^{(i)}\delta^{1/2}\curl(\nu(x_0)\Phi_0(\cdot,z_\delta))$. Let $(\wid U_\delta^{(i)\prime},\wid V_\delta^{(i)\prime})$, $i=1,2$, be the unique solution to problem \eqref{2.1} with the boundary data $\nu\times\eta U_{0,\delta}'$. Utilizing the same method as above, we again obtain a contradiction and thus $\gamma_1=\gamma_2$ on $\Gamma$, which finishes the proof. 
\end{proof}

Since $\mu_1=\mu_2$ and $\gamma_1=\gamma_2$ on $\Gamma$, in the following, for $x_0\in\Gamma$ we still use the notation $\mu_0,\gamma_0$ and $k_0$ in Theorem \ref{thm3.2} (rather than $\mu_0^{(i)},\gamma_0^{(i)}$). Further, for $(\mu_i,\varepsilon_i,D_i,\mathcal{B}_i)$, $i=1,2$, denote by $(E_\delta^{(i)},H_\delta^{(i)})$, $\varphi_{m,\delta}^{(i)},\psi_{m,\delta}^{(i)},\Phi_{m,\delta}^{(i)}$ and $\Psi_{m,\delta}^{(i)}$ the functions defined in Theorem \ref{thm3.2} with $m\in\N\cup\{-1,0\}$ and $\delta\in(0,\delta_0)$. (For functions defined in Theorem \ref{thm3.4}, we apply the same manner for notations.) 
\begin{lemma}\label{lem4.2}
	For any fixed $x_0\in\Gamma$ and $m\in\N$, the following estimates 
	\begin{align*}
	  &\|\varphi_{m,\delta}^{(1)}-\varphi_{m,\delta}^{(2)}\|_{H^{m+1/2}(\pa\Om)}+\|\nu\times(\Phi_{m,\delta}^{(1)}-\Phi_{m,\delta}^{(2)})\|_{H^{m+1/2}_t(\pa\Om)}\leq C, \\
	  &\|\psi_{m,\delta}^{(1)}-\psi_{m,\delta}^{(2)}\|_{H^{m+1/2}_t(\pa\Om)}+\|\nu\cdot(\Psi_{m,\delta}^{(1)}-\Psi_{m,\delta}^{(2)})\|_{H^{m+1/2}(\pa\Om)}\leq C 
	\end{align*}
	hold uniformly for $\delta\in(0,\delta_0)$. 
\end{lemma}
\begin{proof}
	Since $\mu_1=\mu_2$ on $\Gamma$, from Remark \ref{re3.3} we immediately obtain that $\|\nu\times(\Phi_{m,\delta}^{(1)}-\Phi_{m,\delta}^{(2)})\|_{H^{m+1/2}_t(\pa\Om)}$ and $\|\nu\cdot(\Psi_{m,\delta}^{(1)}-\Psi_{m,\delta}^{(2)})\|_{H^{m+1/2}(\pa\Om)}$ are uniformly bounded for $\delta$. Again let $\varepsilon>0$ be such that $\pa\Om\cap B_{2\varepsilon}(x_0)\subset\Gamma$ and $B_{2\varepsilon}(x_0)\cap(D_1\cup D_2)=\emptyset$, and $\eta\in C^\infty(\R^3)$ be the corresponding cut-off function in the proof of Lemma \ref{lem4.1}. Define $(u_\delta^{(i)},v_\delta^{(i)})$, $i=1,2$, to be the unique solution to problem \eqref{2.1} with the boundary data $\nu\times\eta E_{0,\delta}$. Then we know that $\nu\times v_\delta^{(1)}=\nu\times v_\delta^{(2)}$ on $\Gamma$. Further, by the regularity of Maxwell's equation, we deduce that for arbitrarily fixed $m\in\N$ and $i=1,2$, 
	\ben
	  \|u_\delta^{(i)}-E_\delta^{(i)}\|_{H^m(\Om\setminus\ov D_i)}+\|v_\delta^{(i)}-H_\delta^{(i)}\|_{H^m(\Om\setminus\ov D_i)}\leq C
	\enn
	uniformly for $\delta$, which implies $\|\psi_{m,\delta}^{(1)}-\psi_{m,\delta}^{(2)}\|_{H^{m+1/2}_t(\pa\Om)}\leq C$ by Theorem \ref{thm3.2} and Remark \ref{re3.3}. 
	
	Finally we show the estimate about $\varphi_{m,\delta}^{(i)}$. Recalling the functions $(\wid E_\delta^{(i)},\wid H_\delta^{(i)})$ in the proof of Theorem \ref{thm3.2}. In view of \eqref{3.3}, it suffices to prove that for $m\in\N\cup\{0\}$, 
	\be\label{4.2}
	  \|\nu\cdot[\gamma_1(\wid E_\delta^{(1)}-E_{0,\delta})-\gamma_2(\wid E_\delta^{(2)}-E_{0,\delta})]\|_{H^{m+1/2}(\pa\Om)}\leq C
	\en
	uniformly for $\delta$, which clearly holds when $m=0$. From the identity  
	\begin{align*}
	  &\quad\dive_x(\Phi_{k_0}(x,y)\nu(y)\times(\wid H_\delta^{(i)}(y)-H_{0,\delta}(y))) \\
 	  &=\nu(y)\cdot\curl_y(\Phi_{k_0}(x,y)(\wid H_\delta^{(i)}(y)-H_{0,\delta}(y)))-\Phi_{k_0}(x,y)\nu(y)\cdot\curl_y(\wid H_\delta^{(i)}(y)-H_{0,\delta}(y))
	\end{align*}
	and the second equation in \eqref{3.5}, we see that in $\Om$ 
	\ben
	  \mathcal{S}_{k_0}(\nu\cdot\gamma_i(\wid E_\delta^{(i)}-E_{0,\delta}))=\frac{1}{i\om}\dive\mathcal{S}_{k_0}(\nu\times(\wid H_\delta^{(i)}-H_{0,\delta}))-\mathcal{S}_{k_0}(\nu\cdot(\gamma_i-\gamma_0)E_{0,\delta})
	\enn
	with $i=1,2$. Since $\nu\times v_\delta^{(1)}=\nu\times v_\delta^{(2)}$ on $\Gamma$, we obtain that for all $m\in\N$, $\|\nu\times[(\wid H_\delta^{(1)}-H_{0,\delta})-(\wid H_\delta^{(2)}-H_{0,\delta})]\|_{H^{m+1/2}_t(\pa\Om)}\leq C$ uniformly for $\delta$, which indicates that for any fixed $m\in\N$, 
	\ben
	  \|\mathcal{S}_{k_0}(\nu\cdot[\gamma_1(\wid E_\delta^{(1)}-E_{0,\delta})-\gamma_2(\wid E_\delta^{(2)}-E_{0,\delta})])\|_{[H^m(\Om)]^3}\leq C. 
	\enn
	Taking the normal derivative on $\pa\Om$ then yields that for any fixed $m\in\N$
	\be\label{4.3}
	  \|(I+2{K}_{k_0}')(\nu\cdot[\gamma_1(\wid E_\delta^{(1)}-E_{0,\delta})-\gamma_2(\wid E_\delta^{(2)}-E_{0,\delta})])\|_{H^{m+1/2}(\pa\Om)}\leq C. 
	\en
	Now suppose \eqref{4.2} holds true when $m\leq l$ with $l\geq0$. From Theorem \ref{thm2.1} we know  
	\ben
	\|{K}_{k_0}'(\nu\cdot[\gamma_1(\wid E_\delta^{(1)}-E_{0,\delta})-\gamma_2(\wid E_\delta^{(2)}-E_{0,\delta})])\|_{H^{l+3/2}(\pa\Om)}\leq C, 
	\enn
	which implies by \eqref{4.3} that the inequality \eqref{4.2} holds when $m=l+1$. The induction arguments then completes the proof. 
\end{proof}
\begin{lemma}\label{lem4.3}
	For any fixed $x_0\in\Gamma$ and $l\in\N\cup\{0\}$, suppose $D^\alpha\mu_1=D^\alpha\mu_2$ and $D^\alpha\gamma_1=D^\alpha\gamma_2$ on $\Gamma_\varepsilon$ when $|\alpha|\leq l$. Let $q_1,q_2\in C^\infty(\ov\Om)$ be two functions satisfying $D^\alpha q_1=D^\alpha q_2$ on $\Gamma_\varepsilon$ with $|\alpha|\leq l$. It then follows that 
	\begin{align}\label{4.4}
	&\|q_1\Phi_{1,\delta}^{(1)}-q_2\Phi_{1,\delta}^{(2)}\|_{[H^{l+1}(\Om)]^3}+\|q_1\Psi_{1,\delta}^{(1)}-q_2\Psi_{1,\delta}^{(2)}\|_{[H^{l+1}(\Om)]^3}\leq C, \\ \label{4.5}
	&\|D^\alpha q_1\Phi_{1,\delta}^{(1)}-D^\alpha q_2\Phi_{1,\delta}^{(2)}\|_{[H^{l+2-|\alpha|}(\Om)]^3}+\|D^\alpha q_1\Psi_{1,\delta}^{(1)}-D^\alpha q_2\Psi_{1,\delta}^{(2)}\|_{[H^{l+2-|\alpha|}(\Om)]^3}\leq C
	\end{align}
	with $1\leq|\alpha|\leq l+2$ and the positive constant $C$ independent of $\delta\in(0,\delta_0)$. 
\end{lemma}
\begin{proof}
	We first consider the case that $q_1=q_2=1$. By the definition of $\Phi_{1,\delta}^{(i)}$ and $\Psi_{1,\delta}^{(i)}$, we see that 
	\begin{align*}
	  \Phi_{1,\delta}^{(1)}-\Phi_{1,\delta}^{(2)}=&\;i\om\curl\mathcal{G}_{k_0}([\gamma_1(\mu_1-\mu_0)-\gamma_2(\mu_2-\mu_0)]H_{0,\delta}) \\
	  &\;+\na\mathcal{G}_{k_0}(\na(\gamma_1-\gamma_2)\cdot E_{0,\delta}), \\
	  \Psi_{1,\delta}^{(1)}-\Psi_{1,\delta}^{(2)}=&\;i\om\curl\mathcal{G}_{k_0}([\mu_1(\gamma_1-\gamma_0)-\mu_2(\gamma_2-\gamma_0)]E_{0,\delta}) \\
	  &\;+\na\mathcal{G}_{k_0}(\na(\mu_1-\mu_2)\cdot H_{0,\delta}). 
	\end{align*}
	Since $D^\alpha\mu_1=D^\alpha\mu_2$ and $D^\alpha\gamma_1=D^\alpha\gamma_2$ on $\Gamma_\varepsilon$ with $|\alpha|\leq l$, we know that $[\gamma_1(\mu_1-\mu_0)-\gamma_2(\mu_2-\mu_0)]H_{0,\delta}$, $\na(\gamma_1-\gamma_2)\cdot E_{0,\delta}$, $[\mu_1(\gamma_1-\gamma_0)-\mu_2(\gamma_2-\gamma_0)]E_{0,\delta}$ and $\na(\mu_1-\mu_2)\cdot H_{0,\delta}$ are uniformly bounded in $[H^l(\Om)]^3$ for $\delta$. Then \eqref{4.4} follows from Theorem \ref{thm2.1}. 
	
	Now we prove for general $q_1,q_2$. For simplicity, we only show the estimates about $\Psi_{1,\delta}^{(i)}$. From Remark \ref{re3.3}, we know that $\Psi_{1,\delta}^{(i)}$, $i=1,2$, are uniformly bounded in $[H^1(\Om)]^3$ for $\delta$. Hence \eqref{4.4} and \eqref{4.5} hold when $l=0$. Suppose \eqref{4.4} and \eqref{4.5} hold when $l\leq m_0$ with $m_0\geq0$. We shall prove these estimates for $l=m_0+1$. First, it is easily seen that 
	\ben
	  D^\alpha q_1\Psi_{1,\delta}^{(1)}-D^\alpha q_2\Psi_{1,\delta}^{(2)}=D^\alpha(q_1-q_2)\Psi_{1,\delta}^{(1)}+D^\alpha q_2(\Psi_{1,\delta}^{(1)}-\Psi_{1,\delta}^{(2)}). 
	\enn
	Therefore, it suffices to prove that 
	\begin{align}\label{4.6}
	  &\|(q_1-q_2)\Psi_{1,\delta}^{(1)}\|_{[H^{m_0+1}(\Om)]^3}\leq C, \\ \label{4.7}
	  &\|D^\alpha(q_1-q_2)\Psi_{1,\delta}^{(1)}\|_{[H^{m_0+3-|\alpha|}(\Om)]^3}\leq C
	\end{align}
	for $1\leq|\alpha|\leq m_0+3$. Since $\|\Psi_{l,\delta}^{(1)}\|_{[H^1(\Om)]^3}\leq C$, it follows that \eqref{4.7} holds when $|\alpha|=m_0+2,m_0+3$. Next we claim that if \eqref{4.7} holds for $l_0\leq|\alpha|\leq m_0+3$ with $2\leq l_0\leq m_0+2$, then it also holds for $|\alpha|=l_0-1$, which implies \eqref{4.7} is satisfied for $1\leq|\alpha|\leq m_0+3$. For $|\beta|=l_0-1\leq m_0+1$, since $D^\beta q_1=D^\beta q_2$ on $\Gamma_\varepsilon$, we have $\|D^\beta(q_1-q_2)\Psi_{1,\delta}^{(1)}\|_{[H^{m_0-l_0+7/2}(\pa\Om)]^3}\leq C$. Moreover, it is derived that 
	\begin{align*}
	  &\;(\Delta+k_0^2)[D^\beta(q_1-q_2)\Psi_{1,\delta}^{(1)}] \\
	  =&\;\Delta D^\beta(q_1-q_2)\Psi_{1,\delta}^{(1)}+2\na D^\beta(q_1-q_2)\cdot\na\Psi_{1,\delta}^{(1)} \\
	  &\;-D^\beta(q_1-q_2)(-i\om\curl(\mu_1(\gamma_1-\gamma_0)E_{0,\delta})+\na(\na\mu_1\cdot H_{0,\delta}))
	\end{align*}
	in $\Om$. Then \eqref{4.7} with $|\alpha|=l_0+1$ gives $\|\Delta D^\beta(q_1-q_2)\Psi_{1,\delta}^{(1)}\|_{[H^{m_0+2-l_0}(\Om)]^3}\leq C$. From the observation  
	\ben
	\pa_iD^\beta(q_1-q_2)\pa_i\Psi_{1,\delta}^{(1)}=\pa_i(\pa_iD^\beta(q_1-q_2)\Psi_{1,\delta}^{(1)})-\pa_i^2D^\beta(q_1-q_2)\Psi_{1,\delta}^{(1)}
	\enn
	with $i=1,2,3$ and the estimates \eqref{4.7} with $|\alpha|=l_0,l_0+1$, it further follows that $\|\na D^\beta(q_1-q_2)\cdot\na\Psi_{1,\delta}^{(1)}\|_{H^{m_0+2-l_0}(\Om)}\leq C$. Furthermore, since $D^\alpha[D^\beta(q_1-q_2)]=0$ on $\Gamma_\varepsilon$ with $|\alpha|\leq m_0+2-l_0$, we obtain that $D^\beta(q_1-q_2)(-i\om\curl(\mu_1(\gamma_1-\gamma_0)E_{0,\delta})+\na(\na\mu_1\cdot H_{0,\delta}))$ are uniformly bounded in $[H^{m_0+2-l_0}(\Om)]^3$ for $\delta$. Therefore, the regularity of Helmholtz equation confirms the claim. Finally, similarily as above, we can deduce estimate \eqref{4.6}. The proof is thus complete. 
\end{proof}
\begin{lemma}\label{lem4.4}
	Under the conditions in Lemma {\rm \ref{lem4.3}}, we further have 
	\begin{align*}
	  &\|q_1\Phi_{l,\delta}^{(1)}-q_2\Phi_{l,\delta}^{(2)}\|_{[H^{l+1}(\Om)]^3}+\|q_1\Psi_{l,\delta}^{(1)}-q_2\Psi_{l,\delta}^{(2)}\|_{[H^{l+1}(\Om)]^3}\leq C, \\ 
	  &\|D^\alpha q_1\Phi_{l,\delta}^{(1)}-D^\alpha q_2\Phi_{l,\delta}^{(2)}\|_{[H^{l+2-|\alpha|}(\Om)]^3}+\|D^\alpha q_1\Psi_{l,\delta}^{(1)}-D^\alpha q_2\Psi_{l,\delta}^{(2)}\|_{[H^{l+2-|\alpha|}(\Om)]^3}\leq C
	\end{align*}
	with $1\leq|\alpha|\leq l+2$, where the constant $C>0$ is independent of $\delta\in(0,\delta_0)$. 
\end{lemma}
\begin{proof}
	The assertion is trivial when $l=0$, while the case $l=1$ is proved in Lemma \ref{lem4.3}. Combining Lemmas \ref{lem4.2} and \ref{lem4.3}, the left part of the estimates can be verified by the induction arguments analog to the proof of Lemma \ref{lem4.3}. We therefore omit the detailed proof. 
\end{proof}

As for the functions $\varphi_{m,\delta}^{(i)\prime},\psi_{m,\delta}^{(i)\prime},\Phi_{m,\delta}^{(i)\prime}$ and $\Psi_{m,\delta}^{(i)\prime}$ in Theorem \ref{thm3.4}, we have the same conslusions. The proofs follow a completely similar manner and are thus omitted. 
\begin{lemma}\label{lem4.5}
	For any fixed $x_0\in\Gamma$ and $m\in\N$, the following estimates 
	\begin{align*}
	&\|\varphi_{m,\delta}^{(1)\prime}-\varphi_{m,\delta}^{(2)\prime}\|_{H^{m+1/2}(\pa\Om)}+\|\nu\times(\Phi_{m,\delta}^{(1)\prime}-\Phi_{m,\delta}^{(2)\prime})\|_{H^{m+1/2}_t(\pa\Om)}\leq C, \\
	&\|\psi_{m,\delta}^{(1)\prime}-\psi_{m,\delta}^{(2)\prime}\|_{H^{m+1/2}_t(\pa\Om)}+\|\nu\cdot(\Psi_{m,\delta}^{(1)\prime}-\Psi_{m,\delta}^{(2)\prime})\|_{H^{m+1/2}(\pa\Om)}\leq C 
	\end{align*}
	hold uniformly for $\delta\in(0,\delta_0)$. 
\end{lemma}
\begin{lemma}\label{lem4.6}
	Under the conditions in Lemma {\rm \ref{lem4.3}}, we have 
	\begin{align*}
	&\|q_1\Phi_{1,\delta}^{(1)\prime}-q_2\Phi_{1,\delta}^{(2)\prime}\|_{[H^{l+1}(\Om)]^3}+\|q_1\Psi_{1,\delta}^{(1)\prime}-q_2\Psi_{1,\delta}^{(2)\prime}\|_{[H^{l+1}(\Om)]^3}\leq C, \\ 
	&\|D^\alpha q_1\Phi_{1,\delta}^{(1)\prime}-D^\alpha q_2\Phi_{1,\delta}^{(2)\prime}\|_{[H^{l+2-|\alpha|}(\Om)]^3}+\|D^\alpha q_1\Psi_{1,\delta}^{(1)\prime}-D^\alpha q_2\Psi_{1,\delta}^{(2)\prime}\|_{[H^{l+2-|\alpha|}(\Om)]^3}\leq C 
	\end{align*}
	with $1\leq|\alpha|\leq l+2$ and the positive constant $C$ independent of $\delta\in(0,\delta_0)$. 
\end{lemma}
\begin{lemma}\label{lem4.7}
	Under the conditions in Lemma {\rm \ref{lem4.3}}, we further have 
	\begin{align*}
	&\|q_1\Phi_{l,\delta}^{(1)\prime}-q_2\Phi_{l,\delta}^{(2)\prime}\|_{[H^{l+1}(\Om)]^3}+\|q_1\Psi_{l,\delta}^{(1)\prime}-q_2\Psi_{l,\delta}^{(2)\prime}\|_{[H^{l+1}(\Om)]^3}\leq C, \\ 
	&\|D^\alpha q_1\Phi_{l,\delta}^{(1)\prime}-D^\alpha q_2\Phi_{l,\delta}^{(2)\prime}\|_{[H^{l+2-|\alpha|}(\Om)]^3}+\|D^\alpha q_1\Psi_{l,\delta}^{(1)\prime}-D^\alpha q_2\Psi_{l,\delta}^{(2)\prime}\|_{[H^{l+2-|\alpha|}(\Om)]^3}\leq C 
	\end{align*}
	with $1\leq|\alpha|\leq l+2$, where the constant $C>0$ is independent of $\delta\in(0,\delta_0)$. 
\end{lemma}

With all the preceding preparations, we now are at the position to prove the main results of this paper. 
\begin{proof}{(of Theorem \ref{thm2.1})}
    Suppose that there exists a $l\in\N\cup\{0\}$ and a $x_0\in\Gamma$ such that $D^\alpha\mu_1=D^\alpha\mu_2$ and $D^\alpha \gamma_1=D^\alpha\gamma_2$ on $\Gamma$ when $|\alpha|\leq l$, and $\pa_\nu^{l+1}\mu_1(x_0)\neq\pa_\nu^{l+1}\mu_2(x_0)$ or $\pa_\nu^{l+1}\gamma_1(x_0)\neq\pa_\nu^{l+1}\gamma_2(x_0)$. Without loss of genearlity, we assume that $\pa_\nu^{l+1}\mu_1(x_0)\neq\pa_\nu^{l+1}\mu_2(x_0)$. The case for $\gamma$ can be handled similarly using Theorem \ref{thm3.8} and Lemmas \ref{lem4.5}-\ref{lem4.7}. In view of Lemma \ref{lem4.2}, we have the estimate that 
    \ben
      \|\psi_{l+1,\delta}^{(1)}-\psi_{l+1,\delta}^{(2)}\|_{H^{l+3/2}_t(\pa\Om)}+\|\nu\cdot(\Psi_{l+1,\delta}^{(1)}-\Psi_{l+1,\delta}^{(2)})\|_{H^{l+3/2}(\pa\Om)}\leq C. 
    \enn
    Recalling the definition of $\psi_{l+1,\delta}^{(i)}$ and $\Psi_{l+1,\delta}^{(i)}$, applying Lemma \ref{lem4.4} we obtain that 
    \ben
    \|\psi_{1,\delta}^{(1)}-\psi_{1,\delta}^{(2)}\|_{H^{l+3/2}_t(\pa\Om)}+\|\nu\cdot(\Psi_{1,\delta}^{(1)}-\Psi_{1,\delta}^{(2)})\|_{H^{l+3/2}(\pa\Om)}\leq C. 
    \enn
    From the regularity of $\mu$ and $\gamma$ on $\Gamma$, we see that $[\mu_1(\gamma_1-\gamma_0)-\mu_2(\gamma_2-\gamma_0)]E_{0,\delta}$ are uniformly bounded in $H^{l+1}(\Om)$, which implies that 
    \ben
    &&\left\|\nu\times\na G_{k_0}(\na(\mu_1-\mu_2)\cdot H_{0,\delta})\right\|_{H_t^{l+3/2}(\pa\Om)} \\
    &&\qquad+\left\|\nu\cdot\na G_{k_0}(\na(\mu_1-\mu_2)\cdot H_{0,\delta})\right\|_{H^{l+3/2}(\pa\Om)}\leq C 
    \enn
    uniformly for $\delta\in(0,\delta_0)$. However, this contradicts with Theorem \ref{thm3.7} and the proof is thus complete. 
\end{proof}
%\begin{proof}{(of Theorem \ref{thm2.2})}
%	
%\end{proof}
\begin{remark}\label{rem3.8}
	It is clear that our proof after slight modification still work if we are knowing the local impedance map $\Lambda^{-1}$. 
\end{remark}

\section*{Acknowledgements}
This work was supported by the NNSF of China with grant 12122114. 
%Our manuscript has no associated data, and there is no conflict of interest between authors. 

\end{document}